\documentclass[12pt,thmsa]{amsart}%
\usepackage[T1]{fontenc}
\usepackage[utf8]{inputenc}

\usepackage{amsmath}
\usepackage{amsthm}
\usepackage{graphicx}
\usepackage{amsfonts}
\usepackage{amssymb}

\providecommand{\U}[1]{\protect\rule{.1in}{.1in}}

\newtheorem{theorem}{Theorem}
\newtheorem{proposition}{Proposition}
\newtheorem{lemma}[proposition]{Lemma}
\newtheorem{corollary}[proposition]{Corollary}

\theoremstyle{definition}
\newtheorem{definition}[proposition]{Definition}

\theoremstyle{remark}
\newtheorem{remark}{Remark}

\newcommand{\commentaire}[1]{}

\newcommand{\N}{\mathbb N}
\newcommand{\Q}{\mathbb Q}
\newcommand{\R}{\mathbb R}
\newcommand{\Z}{\mathbb Z}
\newcommand{\T}{\mathbb T}

\newcommand{\fff}{\rightarrow}

\newcommand{\SL}{\operatorname{SL}}
\newcommand{\sgn}{\operatorname{sgn}}
\begin{document}
\title{Dynamic of generalized transvections}
\author{Guido Ahumada, Nicolas Chevallier}
\maketitle
\begin{abstract}
Given an increasing odd homeomorphism $\sigma:\R\fff\R$,  the two bijective maps $h_{\sigma},v_{\sigma}:\R^2\fff \R^2$ defined by 
\[
h_{\sigma}(x,y)=(x+\sigma^{-1}(y),y) \text{ and } v_{\sigma}(x,y)=(x,\sigma(x)+y).
\]
are called generalized transvections.
We study  the action on the plane of the group $\Gamma(\sigma)$  generated by these two maps. Particularly interesting cases arise when $\sigma(x)=\sgn(x)|x|^{\alpha}$. We prove that most points have dense orbits and that every nonzero point has a dense orbit when $\sigma(x)=\sgn(x)|x|^{2}$. We also look at  invariant measures and thanks to Nogueira's work about $\SL(2,\Z)$-invariant measure, we can determine these measures when $\sigma$ is linear in a neighborhood of the origin.
\end{abstract}
\section{Introduction}
We are going to study some algebraic and dynamic properties of the group of homeomorphisms generated by two {\it generalized transvections} acting on the plane $\R^2$. Given an increasing odd homeomorphism $\sigma:\R\fff\R$,  these generalized transvections are two bijective maps $h_{\sigma},v_{\sigma}:\R^2\fff \R^2$ defined by 
\[
h_{\sigma}(x,y)=(x+\sigma^{-1}(y),y) \text{ and } v_{\sigma}(x,y)=(x,\sigma(x)+y).
\]
and $\Gamma(\sigma)$ is the group generated by these two maps. 
 
Let us look at some particular cases. When $\sigma$ is the identity map, $\Gamma(\sigma)=\SL(2,\Z)$.  When $\sigma(x)=\sgn(x)|x|^\alpha$ for some positive real number $\alpha$, the maps $h_{\sigma}$ and $v_{\sigma}$ are the simplest ``skew'' homeomorphisms that commute
with the flow induced by the vector field $X(x,y)=(x,\alpha y)$. These maps  also act on the real
weighted projective space defined by the curves $x\fff(x,\lambda \sigma(x))$,
$\lambda\in\R\cup\{\infty\}$. When $\sigma(x+1)=\sigma(x)+1$, $\sigma$ restricted to $\Z$ is the identity, therefore there is no difference in the actions of $\Gamma(\sigma)$ and $\SL(2,\Z)$ on $\Z^2$. Furthermore, using Fubini theorem we see that the Lebesgue measure is invariant by both maps $h_{\sigma}$ and $v_{\sigma}$,
so, in many perspectives, $\Gamma(\sigma)$ can be seen as a perturbed version of $\SL(2,\Z)$ and it is natural to study  how  the properties of $\SL(2,\Z)$ are transformed by such nonlinear  perturbations.   

On the algebraic side, the main question is to find the generators and relations presentation of the group $\Gamma(\sigma)$, in particular is it possible that $\Gamma(\sigma)$ is a free group? We prove only partial results, see section \ref{sec:algebraic}.

Before taking an interest in the dynamical side, let us introduce the {\it generalized} Euclidean algorithm, it is the map $E_{\sigma}$ defined by 
\[
E_{\sigma}:(x,y)\in\R_{\geq 0}^2\fff
\left\{
\begin{array}
[l]{l}
(x-\sigma^{-1}(y),y) \text{ if } y<\sigma(x) \\
(x,y-\sigma(x)) \text{\hspace{0.4cm} if } y\geq \sigma(x)
\end{array}
\right.
\]
(we shall define the Euclidean algorithm on the whole plane, see section \ref{sec:paradox}). When $\sigma$ is the identity, $E=E_{\sigma}$ is the classical Euclidean algorithm map. In that case, it is known that a point  in the plane  has a dense orbit under the natural action of $\SL(2,\Z)$ if and only if it does not belong to a rational line through the origin. It is also known that 
\[
\bigcup_{n=0}^{\infty}\bigcup_{k=0}^{\infty}E^{-k}(\{E^n(\{p\})\})=\R_{\geq 0}^2\cap \SL(2,\Z)p
\]
for all $p\in\R_{\geq 0}^2$ (see \cite{No1}). This latter fact together with the ergodicity of the action of $\SL(2,\Z)$ on $\R^2$, implies the ergodicity of the Euclidean algorithm with respect to the Lebesgue measure. 
We shall prove the following results.
\begin{theorem}\label{thm:density}
Let $\sigma:\R\rightarrow \R$ be an increasing odd homeomorphism. Then all points in $\R^2$ that do not belong to a $\sigma$-rational line have dense orbits in $\R^2$ under the action of $\Gamma(\sigma)$. 
\end{theorem}
See section \ref{sec:density}, 
the definition of $\sigma$-rational lines. We shall also give a version of this Theorem in dimensions $\geq 3$ (see Theorems \ref{thm:densityn1} and \ref{thm:densityn2}).  So at first sight, the general case is similar to the classical case, however,
\begin{theorem}\label{thm:density2}
If $\sigma:x\in\R\fff \sgn(x)x^2$, then every non zero point in $\R^2$ has a dense orbit under the action of  $\Gamma(\sigma)$. Although,  $\bigcup_{n=0}^{\infty}\bigcup_{k=0}^{\infty}E_{\sigma}^{-k}(\{E_{\sigma}^n(1,0)\})$ is a discrete subset of $\R_{\geq 0}^2$.
\end{theorem}

A paradoxical partition of $\Omega=\R^2\setminus\{(0,0)\}$ is naturally associated with the two maps $h_{\sigma}$ and $v_{\sigma}$. Let
\begin{align*}
X&=\{(x,y)\in\Omega:xy\geq 0 \text{ and } x\neq 0\}\\
Y&=\{(x,y)\in\Omega:xy\leq 0 \text{ and } y\neq 0\}.
\end{align*}
The disjoint sets $A=h_{\sigma}(X)$, $B=v_{\sigma}(X)$, $C=h^{-1}_{\sigma}(Y)$ and $D=v^{-1}_{\sigma}(Y)$ satisfy the paradoxical properties
\begin{align*}
\Omega&=A\cup B\cup C\cup D,\\
\Omega&=h^{-1}_{\sigma}(A)\cup h_{\sigma}(C)= v^{-1}_{\sigma}(B)\cup v_{\sigma}(D).
\end{align*}
This implies that there is no finitely additive probability measure defined on the Borel subsets of $\Omega$ and invariant by $h_{\sigma}$ and $v_{\sigma}$. Moreover, this nonexistence can be made quantitative: for any finitely additive probability  measure $\mu$, there exists a Borel set $F\subset\Omega$ such that 
\[
|\mu(h_{\sigma}(F))-\mu(F)| \text{ or } |\mu(v_{\sigma}(F))-\mu(F)|\geq \frac14
\]
 (see \cite{Bu}, \cite{Sh} and section \ref{sec:paradox}).
Concerning infinite $\sigma$-additive measures, Dani proved that up to a multiplicative constant, the Lebesgue measure is the only locally finite invariant measure that does not charge the axes (\cite{Da1}). Nogueira gave a direct proof of this result not using homogeneous dynamic (\cite{No2}) and also described the ergodic measures that charge the axes. Nogueira's direct proof naturally leads to the following result. 

\begin{theorem}\label{thm:uniqueness}
Let $\sigma:\R\rightarrow \R$ be an increasing odd homeomorphism. Suppose $\sigma(x)=ax$ for all $x\in[-b,b]$ for some $a,b>0$. 
Then the Lebesgue measure is the only measure $\mu$ on $\R^2_{\geq 0}$ such that 
\begin{itemize}
\item $\mu([0,1]^2)=1$ and $\mu(\{0\}\times \R_{\geq 0})=\mu(\R_{\geq 0}\times\{0\})=0$,
\item for all Borel subsets $A$ in $\R^2_{\geq 0}$, 
$
\mu(A)=\mu(h_{\sigma}(A))=\mu(v_{\sigma}(A))
$.
\end{itemize}
\end{theorem}
Actually, we use a local version of Nogueira's theorem, see section \ref{sec:nogueira}.

The above Theorem has an easy consequence.
\begin{corollary}\label{cor:ergodic}
With the same hypothesis as in Theorem \ref{thm:uniqueness}, the generalized Euclidean algorithm is ergodic on $\R^2_{\geq 0}$ equipped with the Lebesgue measure.
\end{corollary}
Miernowski and Nogueira showed that the classical Euclidean algorithm is not only ergodic but also exact, \cite{MiNo}. Using their result we obtain
\begin{corollary} \label{cor:exact}
With the same hypothesis as in Theorem \ref{thm:uniqueness}, the generalized Euclidean algorithm is exact on $\R^2_{\geq 0}$ equipped with the Lebesgue measure.
\end{corollary}
As a last remark, it is well known that the classical Euclidean algorithm is totally dissipative, and it is straightforward to check that the generalized versions are still totally dissipative.

The paper is organized as follows. We begin with some additional remarks   and then with some basic facts about the partitions generated by generalized transvections and the non-existence of finite invariance measures. Next in section \ref{sec:algebraic}, we present a few algebraic results. Section \ref{sec:density} is devoted to the proof of Theorem \ref{thm:density} about the density of orbits.
In this section, proposition \ref{prop:guido} on the density of rational lines is the main intermediate result towards Theorem \ref{thm:density}. The accelerated Euclidean algorithm is also an important tool in the proof of Theorem \ref{thm:density} and we explain how it can be used to define a $\sigma$-continued fraction expansion.  Results in higher dimensions are given in section \ref{sec:high}. Section \ref{sec:density2} is devoted to the proof of Theorem \ref{thm:density2}.
Section \ref{sec:abstract} gives an abstract lemma about extension of invariant measure. In Section \ref{sec:nogueira}, we use Nogueira's theorem and the previous lemma to derive Theorem \ref{thm:uniqueness}.
In section \ref{sec:lines}, we study the lines defined by the generalized transvections in connection with the Euclidean algorithm. A construction of invariant measures for the generalized Euclidean algorithm is outlined in Section \ref{sec:invariant}. 
In section \ref{sec:kazdhan}, we consider an affine extension of the group $\Gamma(\sigma)$ and explain that this group satisfies a relative Kazdhan property. In section \ref{sec:torus}, we consider the case of generalized transvections acting on the two dimensional torus, Proposition \ref{prop:torus} generalizes Burger's result about ergodic $\SL(2,\Z)$-invariant measures on the torus (\cite{Bu}, Proposition 9).
The paper ends with some questions.

Sections \ref{sec:density},   \ref{sec:density2} and  \ref{sec:nogueira} devoted to the proof of Theorems \ref{thm:density}, \ref{thm:density2} and  \ref{thm:uniqueness} are the heart of our work.
They used only sections \ref{sec:paradox} and \ref{sec:abstract}. 
 
\section{some additional remarks about the case $\sigma(x)=\sgn(x)|x|^{\alpha}$}\label{sec:context}
Fix a positive real number $\alpha$. Consider the homeomorphism $\sigma:\R\fff\R$ defined by  $\sigma(x)=\sgn(x)|x|^{\alpha}$ and the flow $\Phi_{\alpha}:(t,x,y)\in\R\times\R^2\fff(e^tx,e^{\alpha t}y)\in\R^2$ associated with  the vector field $X(x,y)=(x,\alpha y)$.   

Call $G(\alpha)$ the group of homeomorphisms $g:\R^2\fff\R^2$ that commute with $\Phi_{\alpha}$, i.e., for all $p\in\R^2$ and all $t\in\R$,
\[g\circ\Phi_{\alpha}(t,p)=\Phi_{\alpha}(t,g(p)).
\]  
\begin{enumerate}
\item It is well known that if $g\in G(1)$ is differentiable  then $g$ is linear. Therefore $\SL(2,\R)$ is the subgroup of $G(1)$ made of differentiable homeomorphisms that preserve the Lebesgue measure. 
 
\item The map $f:\R^2\fff\R^2$ defined  by $f(x,y)=(x,\sigma^{-1}(y))$   intertwines the flows $\Phi_1$ and $\Phi_{\alpha}$: for all $p\in\R^2$ and all $t\in\R$,
\[f\circ\Phi_{\alpha}(t,p)=\Phi_{1}(t,f(p)).
\]
 Thus the map $F_{\alpha}:g\in G(\alpha)\fff f\circ g\circ f^{-1}\in G(1)$ is an isomorphism, and $F_{\alpha}(\Gamma(\sigma))$ is a subgroup of $G(1)$ isomorphic to $\Gamma(\sigma)$, the elements of which preserve the image of the Lebesgue measure by $f$. 
\item Let  $+_{\sigma}$ be the addition defined on $\R$ by $x+_{\sigma}x'=\sigma^{-1}(\sigma(x)+\sigma(x'))$. Observe that $(\R,+_{\sigma},\times)$ is still a field and that 
\begin{align*}
&f\circ h_{\sigma}\circ f^{-1}(x,y)=(x+y,y)\\
&f\circ v_{\sigma}\circ f^{-1}(x,y)=(x,y+_{\sigma}x),
\end{align*} 
so that, with new coordinates, $h_{\sigma}$ and $v_{\sigma}$ becomes standart transvections  associated with  different additions. 
\end{enumerate}
 
\section{The paradoxical decomposition and the Euclidean algorithm} \label{sec:paradox}
Given an increasing odd homeomorphism $\sigma :\R\fff\R$, the generalized Euclidean algorithm is defined on the whole $\Omega=\R^2\setminus \{(0,0)\}$ by 
\begin{align*}
E_{\sigma}&:\Omega\fff\Omega\\
&:u\fff 
\left\{
\begin{array}[l]{lll}
h^{-1}_{\sigma}(u) &\text{ if }& u\in A=h_{\sigma}(X)\\
v^{-1}_{\sigma}(u) &\text{ if }& u\in B=v_{\sigma}(X)\\
h_{\sigma}(u)& \text{ if }& u\in C=h^{-1}_{\sigma}(Y)\\
v_{\sigma}(u) &\text{ if }& u\in D=v^{-1}_{\sigma}(Y)
\end{array}
\right.
\end{align*}
The sets $A$, $B$, $C$ and $D$ are very well adapted to show that there is no $\Gamma(\sigma)$-invariant probability measure.
Indeed, if $\mu$ is a finitely additive probability measure on $\Omega$, since $X=A\cup B$, and $Y=C\cup D$,
\begin{align*}
\begin{array}[c]{ll}
\mu(h^{-1}_{\sigma}(A))=\mu(A)+\mu(B),&\mu(h_{\sigma}(C))=\mu(C)+\mu(D), \\
\mu(v^{-1}_{\sigma}(B))=\mu(A)+\mu(B),&\mu(v_{\sigma}(D))=\mu(C)+\mu(D).\\
\end{array}
\end{align*}
Therefore, 
\begin{align*}
1=&\mu(h^{-1}_{\sigma}(A))-\mu(A)+\mu(v^{-1}_{\sigma}(B))-\mu(B)\\
&+\mu(h_{\sigma}(C))-\mu(C)+\mu(v_{\sigma}(D))-\mu(D)
\end{align*}
which implies that one of the four terms of the right hand side is $\geq \frac14$. It follows that no finitely additive probability can be invariant by the group $\Gamma(\sigma)$.

\begin{lemma}\label{lem:free}
The monoid  $\mathcal M_{\sigma}(X)=\mathcal M(h_{\sigma},v_{\sigma})$ generated by $h_{\sigma}$ and $v_{\sigma}$, and the monoid $\mathcal M_{\sigma}(Y)=\mathcal M(h_{\sigma}^{-1},v_{\sigma}^{-1})$ generated by $h_{\sigma}^{-1}$ and $v_{\sigma}^{-1}$, are free.
\end{lemma}
\begin{proof}
The two nonempty sets $A=h_{\sigma}(X)$ and $B=v_{\sigma}(X)$ form a partition of $X$. 
It follows that for any nontrivial word $w=w_1\circ w_2\dots \circ w_n$ with $w_i=h_{\sigma}$ or $v_{\sigma}$ for each $i$, $w(X)$ is included in $A$ or $B$ and never equal to $X$. Therefore $w\neq Id$ and the monoid $\mathcal M(h_{\sigma},v_{\sigma})$ is free. 
Replacing $A$ and $B$ by $C$ and $D$, we see that the monoid $\mathcal M(h_{\sigma}^{-1},v_{\sigma}^{-1})$ is free as well.
\end{proof}
An easy induction starting with the partition $X=h_{\sigma}(X)\cup v_{\sigma}(X)$ leads to
\begin{lemma}
For each positive integer $n$ we have the partition
\[
X=\bigcup_{w\in \mathcal M_{\sigma,n}(X)} w(X)
\] 
where $\mathcal M_{\sigma,n}(X)$ is the set of words of length $n$  in $\mathcal M_{\sigma}(X)$. Similar partition of $Y$ are given by the words of length $n$ in the monoid $\mathcal M_{\sigma}(Y)$.
\end{lemma}  

\begin{corollary} \label{cor:euclide}
For each $p\in X$ and for each positive integer $n$, there exists a unique word $w_n\in \mathcal M_{\sigma}(X)$ such that $p\in w(X)$. Moreover, $E_{\sigma}^n(p)=w_n^{-1}(p)$. 
A similar result holds for points in $Y$. 
\end{corollary}

\begin{corollary}
For each $p\in X$  there exists a unique infinite word $w=w_1w_2\dots w_n\dots$ written with the alphabet $\{h_{\sigma},\,v_{\sigma}\}$ such that 
\[
p\in \bigcap_{n\geq 1}w_1\circ w_2\dots\circ w_n(X).
\]  
\end{corollary}

\section{Algebraic properties of $\Gamma(\sigma)$}\label{sec:algebraic}

In the previous section we observed that the monoid   $\mathcal M_{\sigma}(X)=\mathcal M(h_{\sigma},v_{\sigma})$ is free (see Lemma \ref{lem:free}).  
Now we are interested in the algebraic structure of the group $\Gamma(\sigma)$ generated by $h_{\sigma}$ and $v_{\sigma}$. This structure is more difficult to apprehend than for the monoid. We have only two partial results. 
 
Let denote $\Gamma_2(\sigma)$ the group generated by 
$h_{\sigma}^2$ et $v_{\sigma}^2$.

\begin{proposition}
$\Gamma_{2}(\sigma)$ is a free group
\end{proposition}

\begin{proof}

 The nonempty subsets $P=C\cup A$ and $Q=B\cup D$
in $\R^{2}$ are disjoint and it is easy to check that $v_{\sigma}^{2k}(P)\subset Q$  and
$h_{\sigma}^{2l}(Q)\subset P$ for all nonzero integers $k$ and $l$. Making use of the ping-pong Lemma 
we see obtain that a non trivial product of power of  $h_{\sigma}^2$ and $v_{\sigma}^2$
cannot be the identity map.
\end{proof}

\begin{remark}
The groups $\Gamma(\sigma)$ and $\Gamma(\sigma^{-1})$ are conjugate : 
$S^{-1}\circ h_{\sigma}\circ S:(x,y)\rightarrow(x+\sigma(x),y)$ and $S^{-1}\circ v_{\sigma}\circ
S:(x,y)\rightarrow(x,y+\sigma^{-1}(x))$.
where $S:(x,y)\rightarrow(y,x)$
\end{remark}
\begin{remark}
If $\sigma=a\,Id$ for some $a$, then $\Gamma
(\sigma)$ and $\Gamma(Id)$ are conjugate subgroups of the linear group.
\end{remark}

\begin{remark}
The central symmetry $(x,y)\rightarrow-(x,y)$ commute with all the elements in
$\Gamma(\sigma)$ because the homeomorphism $\sigma$ is odd.
\end{remark}

\begin{proposition}
Suppose that $\sigma(x)=x$ for all $x$ in a neighborhood of $0$ or that $\sigma(x)=x$ for all $x\in \Z$.
\begin{enumerate}
\item Then there is a surjective morphism $\pi:\Gamma(\sigma)\fff\Gamma(Id)$.

\item  The kernel of $\pi$ is the invariant subgroup generated by  $V_{\sigma}^{2}U_{\sigma}^{3}$ and $V_{\sigma}^{4}$
where $U_{\sigma
}=v_{\sigma}^{-1}h_{\sigma}$ and $V_{\sigma}=h_{\sigma}^{-1}v_{\sigma}h_{\sigma}^{-1}$. 
\end{enumerate}
\end{proposition}

\begin{proof}

1. Let $\mathbb F(a,b)$ be the free group with two generators $a$ and $b$ and let $G$ be a group generated by two of its elements $x$ and $y$.
For $w\in\mathbb F(a,b)$ call $w(x,y)$  the image of $w$ by the unique morphism that sends $a$ on $x$ and $b$ on $y$. 

For all $g$  in $\Gamma(\sigma)$ there exists $w\in\mathbb F(a,b)$ such that $g=w(h_{\sigma},v_{\sigma})$. Set $\pi(g)=w(h,v)$ which is an element of $\langle h,v\rangle=\Gamma(Id)$.  In order to see that $\pi$ define a map from
$\Gamma(\sigma)$ to $\Gamma(1)$ it is enough to see that if $w$ and $w'$ are two elements in $\mathbb{F}(a,b)$ such that $w(h_{\sigma},v_{\sigma})=w'(h_{\sigma},v_{\sigma})$ then $w(h,v)=w'(h,v)$.  

If $\sigma(x)=x$ in a neighborhood of $0$, the maps
$w(h_{\sigma},v_{\sigma})$ and $w(h,v)$ are equal in a small enough
neighborhood $W$ of the origin and the same holds for $w'$ with another $W'$. Now $W\cap W'$ generate the vector space $\R^2$ hence the linear maps $w(h,v)$and $w'(h,v)$ are equal. 

If $\sigma(x)=x$ for all $x\in\Z$, then $h_{\sigma}=h$ and $v_{\sigma}=v$ on $\Z^2$. Therefore for all $w$, $w(h_{\sigma},v_{\sigma})=w(h,v)$ on $\Z^2$. Again, since $\Z^2$ generates the vector space $\R^2$, $w(h_{\sigma},v_{\sigma})=w'(h_{\sigma},v_{\sigma})$ implies $w(h,v)=w'(h,v)$.

The same way of reasoning shows that $\pi
(g^{-1})=\pi(g)^{-1}$ and $\pi(g_{1}g_{2})=\pi(g_{1})\pi(g_{2}).$\medskip

2. Since $U_{\sigma}V_{\sigma}=h_{\sigma}^{-1}$,  $U_{\sigma}$ and $V_{\sigma}$ still generate the
group $\Gamma(\sigma)$. Let denote $U=U_{\sigma}$ and $V=V_{\sigma}$  when $\sigma=Id$.  \newline
The group
$\Gamma(Id)=\SL(2,\mathbb{Z)}$ has the presentation $\langle U,V|V^{2}%
U^{3},V^{4}\rangle$ (see \cite{Se}), which means that the kernel of the morphism $w\in\mathbb F(a,b)\fff w(U,V)\in\Gamma(Id)$ is  the invariant subgroup $N$ generated by $a^3b^2$ and $b^4$.

Observe that the relation $\pi(w(U_{\sigma},V_{\sigma}))=w(U,V)$ still holds for all $w\in\mathbb F(a,b)$.

It follows that each element of the invariant subgroup of
$\Gamma(\sigma)$ generated by $V_{\sigma}^{2}U_{\sigma}^{3}$ and
$V_{\sigma}^{4}$ is in the kernel $\pi$ because it is equal to the identity map on a subset that generates the vector space $\R^2$.

Conversely, let $w$ be in $\mathbb F(a,b)$ and
 $g=w(U_{\sigma}V_{\sigma})$ be an element in $\Gamma(\sigma)$ such that $\pi(g)=Id$.  Since $Id=\pi(w(U_{\sigma},U_{\sigma}))=w(U,V)$, $w$ is in the normal subgroup $N$.
\end{proof}

\begin{remark}
In the case of $\SL_{2}(\Z)$,   $U=v^{-1}h$ has order $6$ and
$V=h^{-1}vh^{-1}$ has order $4$ and  $V^{2}=-I=U^{3}$.
Note that if $\sigma\neq Id$ the maps
$V_{\sigma}^{2}U_{\sigma}^{3}$ and $V_{\sigma}^{4}$ are not the identity map of the plan.

\end{remark}

\section{The accelerated Euclidean algorithm, proof of Theorem \ref{thm:density}}\label{sec:density}

There are two versions of the continued fraction algorithm, the subtractive algorithm and the multiplicative algorithm.  In the same way, there are two Euclidean algorithms. We have already defined the subtractive version (the slow version). 
The accelerated Euclidean algorithm $F_{\sigma}$ is defined on $X$  by\medskip\\
- $F_{\sigma}(p)=h_{\sigma}^{-n}(p)$ when $p\in h_{\sigma}(X)=A$ and $n$ is the smallest integer such that $h_{\sigma}^{-n}(p)\in B=v_{\sigma}(X)$ and,\\
-  $F_{\sigma}(p)=v_{\sigma}^{-n}(p)$ when $p\in v_{\sigma}(X)=B$ and $n$ is the smallest integer such that $v_{\sigma}^{-n}(p)\in A$.\medskip\\
The accelerated algorithm $F_{\sigma}$ is defined on $Y$ by similar formulas. 
Up to the axes and the diagonals, the map $F_{\sigma}$ interchange $A$ and $B$, and $C$ and $D$, while $U_{\sigma}=F_{\sigma}^2$ leaves $A$, $B$, $C$ and $D$ invariant.

\begin{definition}
The  $\sigma$-rational lines are the subsets $L$ in $\R^2$ of the following form:
\begin{itemize}
	\item  $L=Ox$ or $Oy$, one of the coordinate axes,
	\item  $L=m(h_{\sigma},v_{\sigma})(Ox)$ where $m(h_{\sigma},v_{\sigma})$ is an element of the monoid
	$\mathcal M_{\sigma}(X)=\mathcal M(h_{\sigma},v_{\sigma})$ generated by $h_{\sigma}$ and $v_{\sigma}$, 
	\item  $L=m(h_{\sigma}^{-1},v_{\sigma}^{-1})(Oy)$ where $m(h_{\sigma}^{-1},v_{\sigma}^{-1})$ is an element of the monoid
	$\mathcal M_{\sigma}(Y)=\mathcal M(h_{\sigma}^{-1},v_{\sigma}^{-1})$ generated by $h_{\sigma}^{-1}$ and $v_{\sigma}^{-1}$, 
\end{itemize}
A point in $\R^2$ is call $\sigma$-rational if it is in a $\sigma$-rational line and $\sigma$-irrational if it is not. 	
\end{definition}
\begin{remark} \label{rem:rational}
The set of $\sigma$-rational lines we defined is not invariant by the action of the group $\Gamma(\sigma)$ and we could have been defined the set of $\sigma$-rational lines as the set of lines $g(Ox)$, $g\in\Gamma(\sigma)$. This latter set is in general strictly larger than the set we defined and so Theorem \ref{thm:density} is stronger with our definition (see remark \ref{rem:rational2}).
\end{remark}
\begin{remark}
In the above definition, the axes $Ox$ and $Oy$ can be used indifferently because $v_{\sigma}(Ox)=h_{\sigma}(Oy)$. 
\end{remark}
\begin{remark}
Each rational line has zero Lebesgue measure because it is the image of a set of zero Lebesgue measure. And since there are only countably many $\sigma$-rational lines, the set of $\sigma$-rational points has zero Lebesgue measure. 
\end{remark}
Next Proposition is an important step toward Theorem \ref{thm:density}. This proposition is very well known in the classical case. In the general case its proof is not a direct adaptation of the classical case. 
The first step of this proof is the following Lemma which gives a rough estimate of the convergence to zero of sequences $(U_{\sigma}^n(p))_n$ with $p\in X$. The proof of this Lemma is easy and follows the classical line as most parts of the proof   of Theorem \ref{thm:density}. 
\begin{lemma}\label{lem:1/2}
If $p\in \Omega$ is a $\sigma$-irrational point then $\|U_{\sigma}^n(p)\|\leq \tfrac{1}{2^n}\|p\|$ where $\|(x,y)\|=|x|+|y|$.
\end{lemma}
\begin{proof}
We can assume that $p=(x,y)\in A=h_{\sigma}(X)$ w.l.g. Since $p$ is $\sigma$-irrational, the coordinates of $E^n(p)$ are both $>0$ for all $n\geq 0$. Hence,  $F_{\sigma}(p)=h_{\sigma}^{-k}(x,y)=(x',y')$ for some positive integer $k$ and  $0<x'<\sigma^{-1}(y)$,
and, 
$
F_{\sigma}(x',y')=v_{\sigma}^{-k'}(x',y')=(x'',y'')
$
for some positive integer $k'$ and  $y''<\sigma^{-1}(x')$. It follows that 
\begin{align*}
&x=x'+k\sigma^{-1}(y) \text{ and } y=y'=y''+k'\sigma(x'),\\
&x''=x'\leq k\sigma^{-1}(y) \text{ and } y''<k'\sigma(x'), 
\end{align*}
which implies
\begin{align*}
x+y=x''+k\sigma^{-1}(y) + y''+k'\sigma(x')\geq 2(x''+y'').
\end{align*}
Therefore $\|U_{\sigma}(p)\|\leq \tfrac{1}{2}\|p\|$ and the Lemma follows by induction.
\end{proof}
\begin{proposition} \label{prop:guido} The union of the rational lines is dense in $\R^2$. 
\end{proposition}
\begin{proof}
Let us prove that $R=\bigcup_{w\in \mathcal M_{\sigma}(X)}w(Ox)$ is dense in $X$. The density in $Y$ is similar.

Suppose on the contrary that there exists an open  set $V$ in $X$ that does not meet $R$. We can suppose that $V$ is connected and bounded. Choose $p$ in $V$ and denote $r=\sup\{d(x,O):x\in V\}$ ($O$ denote the origin of $\R^2$). 
By Lemma \ref{lem:1/2}, we can choose a integer $n$ large enough in order that 
\[
U_{\sigma}^n(V)\subset B(0,r) \text{ for some }r \text{ and }\lambda(V)>\lambda(B(0,r)).
\] 
By definition, $U_{\sigma}^n(p)$ is of the form $U_{\sigma}^n(p)=w_1\circ w_2\dots \circ w_N(p)$ where $w_i=h_{\sigma}^{-1}$ or $v_{\sigma}^{-1}$. 
Consider the map $w=w_N^{-1}\dots \circ w_1^{-1}$, it is an homeomorphism and it preserves the Lebesgue measure. 
The set $w(X)$ contains  $p$ and its boundary is contained in $R$. 
Therefore the boundary of $w(X)$ does not meet $V$. Since $p\in V\cap w(X)$ and since $V$ is connected, $V\subset w(X)$. By Corollary \ref{cor:euclide}, it implies that $w^{-1}(q)=E_{\sigma}^N(q)=U_{\sigma}^n(q)$ for all $q\in V$. Therefore, $w^{-1}(V)=U_{\sigma}^n(V)\subset B(0,r)$. Now $w^{-1}$ preserves the Lebesgue measure, hence $\lambda(V)=\lambda(w^{-1}(V))\leq \lambda(B(0,r))<\lambda(V)$, a contradiction.
\end{proof}
\begin{lemma}\label{lem:density1}
Let $(p_n=(x_n,y_n))_{n\geq 0}$ be a sequence in $\R^2$. Suppose that $(p_n)_{n\geq 0}$ tend to zero and that   all the $x_n$ or all the $y_n$ have the same sign  and are not zero.   Then the closure
\[
\{w(p_n):n\in \N \text{ and } w\in \mathcal M_{\sigma}(X)\}
\]
contains a half axis:\\
$Ox_+$ if all the $y_n$ are $>0$, $Ox_-$ if all the $y_n$ are  $<0$,\\
$Oy_+$ if all the $x_n$ are  $>0$ and $Oy_-$ if all the $x_n$ are  $<0$.
\end{lemma}
\begin{proof}
The four cases are similar and we deal only with the case, $y_n>0$ for all $n$. For $n\in\N$, consider the sets  
\[
A_n=\{h_{\sigma}^k(x_n,y_n)=(x_n+k\sigma^{-1}(y_n),y_n):k\in \N\}.
\]
Observe that the points in $A_n$ are on the line $y=y_n$ and are spaced by intervals of same length $\sigma^{-1}(y_n)$.
Therefore, since the sequences $(x_n)_{n\geq 0}$, $(y_n)_{n\geq 0}$ and   $(\sigma^{-1}(y_n))_{n\geq 0}$ tend to zero, the closure of the union $\bigcup_{n\geq 0}A_n$ contains $Ox_+$. Since all the maps $h_{\sigma}^k$, $k\in\N$, are in $\mathcal M_{\sigma}(X)$ we are done.
\end{proof}
\begin{proposition}
If $p\in X_{+}=\{(x,y)\in X: x>0\}$ is  $\sigma$-irrational, then 
\[\bigcup_{n\geq 0}\bigcup_{k\geq 0}E_{\sigma}^{-n}\{E_{\sigma}^{k}(\{p\})\}
\]
 is dense in $X_{+}$. 
\end{proposition}
\begin{remark}
Similar statements hold for the three other quadrants.
\end{remark}
\begin{remark}\label{rem:closure}
We shall use the following general fact: if $\mathcal M$ is a monoid of continuous maps acting on topological space $Z$ and if the closure of an orbit $\mathcal MA$ of some subset $A\subset Z$, contains a subset $B\subset Z$, then this closure contains $\mathcal MB$.

\end{remark}

\begin{proof}[Proof of the Proposition]
By Lemma \ref{lem:1/2}, the accelerated algorithm generate a sequence $(p_n=(x_n,y_n)=U_{\sigma}^n(p))_{n\geq 0}$ that satisfies the assumptions of the previous Lemma. Since $y_n>0$ for all $n$, the half axis $Ox_+$ is in the closure $\mathcal C$ of 
\[
\{w(p_n):n\in \N \text{ and } w\in \mathcal M_{\sigma}(X)\}
\]
By  proposition \ref{prop:guido}, the rational lines are dense in $\R^2$ and hence $\bigcup_{w\in \mathcal M_{\sigma}(X)}w(Ox_+)$ is dense in $X_+$. Therefore,  $\mathcal C$ contains $X_+$.

It remains to see that $v=w(U_{\sigma}^n(p))\in \bigcup_{k\geq 0}\bigcup_{n\geq 0}E_{\sigma}^{-k}\{E_{\sigma}^{n}(\{p\})\}$ for all $n\in\N$ and  $w\in\mathcal M_{\sigma}(X)$. 
 First observe that $u=U_{\sigma}^n(p)\in \bigcup_{i\geq 0}E_{\sigma}^{i}(\{p\})$. Next $w$ is a word in the free monoid $\mathcal M_{\sigma}(X)$ with a certain lengh $j$. Since $v=w(u)\in w(X)$,  $E_{\sigma}^j(v)=u$ by Corollary \ref{cor:euclide}, and hence $v\in E_{\sigma}^{-j}(\bigcup_{i\geq 0}E_{\sigma}^{i}(\{p\})$. 
\end{proof} 
\begin{proof}[Proof of Theorem \ref{thm:density}] Let $p$ be a $\sigma$-irrational point in $\R^2$. 
We want to prove that its orbit under the action of $\Gamma(\sigma)$ is dense in $\R^2$. 
The point $p$ is in one of the four quadrants, say $p\in X_{+}$. 
By the previous Proposition, we already know that the orbit $p$ is dense in $X_{+}$, hence the half axis $Oy_+$ is in the closure of the orbit of $p$.  Since the rational lines are dense, the set $\mathcal M_{\sigma}(Y)Oy_+$ is dense in $Y_+=\{(x,y)\in Y: y\geq 0\}$ and the orbit of $p$ as well. Continuing this process, turning around the origin, we see that the orbit $p$ is dense in the four quadrants.
\end{proof}

\subsection{$\sigma$-continued fraction expansion.} We suppose that the homeomorphism $\sigma$ is defined by $\sgn(x)|x|^{\alpha}$ where $\alpha$ is a positive real number. Using the property $\sigma(xy)=\sigma(x)\sigma(y)$, we are going to see that the accelerated $\sigma$-Euclidean algorithm  defines a $\sigma$-continued fraction expansion. 

Let$(x,y)=(r\sigma^{-1}(y),y)$ be in $\{(x,y):0< \sigma^{-1}(y)<x\}$. The real number $r\in]1,+\infty[$ is the $\sigma$-(inverse)slope of $(x,y)$.  Suppose that $U_{\sigma}(x,y)=v_{\sigma}^{-b}\circ h_{\sigma}^{-a}(x,y)$ for some positive integers $a$ and $b$. Put $U_{\sigma}(x,y)=(X,Y)=(\sigma^{-1}(Y)r,Y)$.  Since 
\begin{align*}
(\sigma^{-1}(y)r,y)&=h_{\sigma}^a\circ v_{\sigma}^b(\sigma^{-1}(Y)R,Y)\\
&=h_{\sigma}^a(\sigma^{-1}(Y)R,Y(1+b\sigma(R))\\
&=(\sigma^{-1}(Y)(R+a\sigma^{-1}(1+b\sigma(R))),Y(1+b\sigma(R))\\
&=(\sigma^{-1}(y)\tfrac{R+a\sigma^{-1}(1+b\sigma(R))}{\sigma^{-1}(1+b\sigma(R))},Y(1+b\sigma(R)),
\end{align*} 
we have, 
\[
r=\frac{R+a\sigma^{-1}(1+b\sigma(R))}{\sigma^{-1}(1+b\sigma(R))}=a+\frac{1}{\sigma^{-1}(b+\frac{1}{\sigma(R)})}=S_{a,b}(R).
\]
Iterating the accelerated $\sigma$-Euclidean algorithm we obtain sequences  $(r_n)_{\geq 0}$, $(y_n)_{\geq 0}$, $(a_n)_{n\geq 1}$ and $(b_n)_{n\geq 1}$  defined by $r_0=r$ and $y_0=y$ and
\[
(\sigma^{-1}(y_{n+1})r_{n+1},y_{n+1}))=U_{\sigma}(\sigma^{-1}(y_{n})r_{n},y_{n})=v_{\sigma}^{-b_n}\circ h_{\sigma}^{-a_n}(\sigma^{-1}(y_{n})r_{n},y_{n}).
\] 
This gives rise to a $\sigma$-continued fraction expansion of the $\sigma$-slope $r$:
\[
r=S_{a_1,b_1}\circ\dots\circ S_{a_n,b_n}(r_n).
\]
In this context, the ``golden $\sigma$-slope'' can be defined by the equation
\[
r=S_{1,1}(r)=1+\frac{1}{\sigma^{-1}(1+\frac{1}{\sigma(r)})}.
\]
When $\alpha=2$ this latter equation is 
\[
r=1+\frac{1}{\sqrt{1+\frac{1}{r^2}}}
\]
which have an unique solution in the interval $]1,+\infty[$:
\[
r=\frac12\left(1+\sqrt 2 +\sqrt{2\sqrt 2-1}\right)=1.883203506\dots
\]
Actually $r$ is an algebraic integer solution of the equation
\[
r^4-2r^3+r^2-2r+1=0.
\]
\section{Density: higher dimensional results} \label{sec:high}
We fix an odd increasing homeomorphism $\sigma:\R\fff\R$. When $n\geq 3$,
the generalized transvections in $\R^n$ associated with $\sigma$ are  maps from $\R^n$ to $\R^n$ defined as follow. If $i<j\in\{1,\dots,n\}$ are two integers, we define 
\begin{align*}
h_{i,j}(x_1,\dots,x_n)&=(x_1,\dots,x_{i-1},x_i+\sigma^{-1}(x_j),x_{i+1},\dots,x_n),\\
v_{i,j}(x_1,\dots,x_n)&=(x_1,\dots,x_{j-1},x_j+\sigma(x_i),x_{j+1},\dots,x_n).
\end{align*}
Denote $\mathcal M(\sigma,n)$  the monoid generated by the transvections $h_{i,j}$ and $v_{i,j}$, $1\leq i<j\leq n$ and $\Gamma(\sigma,n)$ the group generated by the same transvections.
There are two main results about density. We prove only the second one which is perhaps more surprising. Its proof is an adaptation of a proof of Dany and Nogueira in \cite{DaNo2}.  The proof of the first Theorem is easier. 
\begin{theorem}\label{thm:densityn1}
Suppose $n\geq 3$. Let $x=(x_1,\dots,x_n)$ be in $\R^n$. Suppose that there exists $i<j\in\{1,\dots,n\}$ such that $(x_i,x_j)$ is $\sigma$-irrational in $\R^2$. Then $\Gamma(\sigma,n)x$ is dense in $\R^n$. 
\end{theorem}
\begin{theorem}\label{thm:densityn2}
Suppose $n\geq 3$. Let $x=(x_1,\dots,x_n)$ be in $\R^n$. Suppose that there exists $i<j\in\{1,\dots,n\}$ such that $x_ix_j<0$ and $(x_i,x_j)$ is $\sigma$-irrational in $\R^2$. Then $\mathcal M(\sigma,n)x$ is dense in $\R^n$.
\end{theorem}
We shall use a $n$-dimensional extension of Lemma \ref{lem:density1} combined with the Proposition about the density of rational lines.
\begin{lemma}
Let $1\leq i<j\leq n$ be two integers and let $B$ be  a subset in $\R^n$  such that
\begin{itemize}
\item for all $x=(x_1,\dots,x_n)\in B,\, x_ix_j<0$,
\item The closure of $B$ contains a point $p=(p_1,\dots,p_n)$ with $p_i=p_j=0$. 
\end{itemize} 
Then the closure $C$ of $\mathcal M(\sigma,n)B$ contains
\[
Q_{i,j}=\{(x_1,\dots,x_n)\in \R^n: x_ix_j>0\}.
\] 
\end{lemma}
\begin{proof}[Proof of the Lemma]
We can suppose that $x_i<0$ and $x_j>0$ for all $x=(x_1,\dots,x_n)\in B$ w.l.g.
Denote $(e_1,\dots,e_n)$ the standard basis of $\R^n$.  For integers $l\neq k$ and real numbers $\alpha$, consider the set 
\begin{align*}
A(l,k,\alpha)=\N\tau(\alpha)e_k,
\end{align*}
where $\tau=\sigma$ if $l<k$ and $\tau=\sigma^{-1}$ if $l>k$.
Fix an integer $k\neq i$ and $j$. If $x=(x_1,\dots,x_n)\in B$, then 
$x+A(i,k,x_i)$ and $x+A(j,k,x_j)$ are included in $\mathcal M(\sigma,n)B$. Observe that if $x=(x_1,\dots,x_n)\in B$, $x_i$ and $x_j$ are of opposite sign and so are $\tau(x_i)$ and $\tau'(x_j)$ for any increasing odd homeomorphisms $\tau$ and $\tau'$.   Letting $x\fff p$ in $B$, and  making use of the continuity of $\sigma$ and $\sigma^{-1}$ at zero, we see that  the closure of
\[
B'=\bigcup_{x=(x_1,\dots,x_n)\in B}(x+A(i,k,x_i))\cup(x+A(j,k,x_j))
\]
contains $p+\R e_k$. Therefore $B'$ satisfies
\begin{itemize}
\item $B'\subset\mathcal M(\sigma,n)B$,
\item for all $x=(x_1,\dots,x_n)\in B',\, x_i<0,\,x_j>0$,
\item the closure of $B'$ contains the set $p+\R e_k$. 
\end{itemize}
Using inductively this process with the integer $k$ ranging in $\{1,\dots,n\}\setminus\{i,j\}$, we see that $\mathcal M(\sigma,n)B$ contains a set $B''$ satisfying 
\begin{itemize}
\item $B''\subset\mathcal M(\sigma,n)B$,
\item for all $x=(x_1,\dots,x_n)\in B'',\, x_i<0,\,x_j>0$,
\item The closure of $B''$ contains the set $p+\bigoplus_{k\neq i,j}\R e_k=\bigoplus_{k\neq i,j}\R e_k$. 
\end{itemize}  
Again the sets $x+A(i,j,x_i)$ and $x+A(j,i,x_j)$ are included in $\mathcal M(\sigma,n)B$ for all $x$ in $B''$. Letting $x\fff q$, where $q$ is any point in $\bigoplus_{k\neq i,j}\R e_k$, we see that $C$ contains
\[
\bigoplus_{k\neq i,j}\R e_k+\R_{\geq 0}e_i \text{ and } \bigoplus_{k\neq i,j}\R e_k+\R_{\leq 0}e_j
\]
Finally, we deduce from the density of the rational lines in $\R^2$ that for each $q\in\bigoplus_{k\neq i,j}\R e_k$, the closure of $\mathcal M(h_{i,j},v_{i,j})(q+
\R_{\geq 0}e_i)$ is $q+
\R_{\geq 0}\,e_i+\R_{\geq 0\,}e_j$ and the closure of $\mathcal M(h_{i,j},v_{i,j})(q+
\R_{\leq 0}e_j)$ is $q+
\R_{\leq 0}\,e_i+\R_{\leq 0\,}e_j$. 
Therefore $C$ contains $Q_{i,j}$
\end{proof}

\begin{proof}[Proof of Theorem \ref{thm:densityn2}]
Let $x=(x_1,\dots,x_n)$ be in $\R^n$. Suppose that there exists $i<j\in\{1,\dots,n\}$ such that $x_ix_j<0$ and $(x_i,x_j)$ is $\sigma$-irrational in $\R^2$.

Consider the accelerated Euclidean algorithm in the plane $x+\R e_i+\R e_j$ (we extend this map by not changing the other components). The positive iterates of this map applied to $x$ form a set $B$ that satisfies the assumptions of the above Lemma. Therefore the closure $F$ of $\mathcal M(\sigma,n)B=\mathcal M(\sigma,n)x$ contains $Q_{i,j}$.

For each $k\neq l$, 
\[
B_{k,l}=\{(x_1,\dots,x_n)\in Q_{i,j}: x_k<0 \text{ and } x_l>0\}
\]
satisfies the hypothesis of the above Lemma with the indices $k$ and $l$, therefore $Q_{k,l}\subset F$. Since it holds for any $k,l$, we are done.
\end{proof}

\section{Proof of Theorem \ref{thm:density2}} \label{sec:density2}
Recall that we suppose $\sigma(x)=\sgn(x)x^2$ in Theorem \ref{thm:density2}.
We need two algebraic Lemmas. 
\begin{lemma}
Let $K\subset \R$ be a field, and let $x\leq 0$ and $y\neq 0$ be two elements in $K$. If $z$ is a positive real number which is not a square in $K$ then $x+y\sqrt{z}$ is not a square in $K(\sqrt{z})$. Therefore $K(\sqrt{z})(\sqrt{x+y\sqrt{z}})$ is an extension of degree to of $K(\sqrt{z})$. 	
\end{lemma}	
\begin{proof}
For all $a,b$ in $K$, $(a+b\sqrt{z})^2=a^2+b^2z+2ab\sqrt{z}$ where $a^2+b^2z>0$ unless $a=b=0$. Therefore, since $x\leq 0$ and $y\neq 0$,  $x+y\sqrt{z}$ cannot be a square in $K(\sqrt{z})$. It follows that $X^2-(x+y\sqrt{z})$ is irreducible over $K(\sqrt{z})$.
\end{proof}
\begin{lemma}
Let $x_0$ and $y_0$ be two positive real numbers. Assume that
\begin{enumerate}
	\item[i ] $y_0$ is not a square in $K_0=\Q$,
	\item[ii ] $x_0=a_0+b_0\sqrt{y_0}$ where  $a_0\neq 0$ and $|b_0|\geq 1$ are rational numbers,
	\item[iii ] $y_0>x_0^2$.	
\end{enumerate}
Then $M_0=(x_0,y_0)$ is $\sigma$-irrational. In particular $(-1+\sqrt{2},2)$ is $\sigma$-irrational. 
\end{lemma}

\begin{proof}
Let us proof that the sequence $M_n=(x_n,y_n)$ generated by the accelerated Euclidean algorithm starting with $M_0$, is not stationary. It will ensure that $M_0$ is $\sigma$-irrational. 
The sequence $(y_n)_n$ define a sequence of fields 
\[
K_0=\Q,\,
K_1=K_0(\sqrt{y_0}),\dots,K_n=K_{n-1}(\sqrt{y_{n-1}
}),\dots
\]
Let us proof by induction that for all $n$:
\begin{enumerate}
	\item[i(n)] $y_n\in K_n$ is positive and is not a square in $K_n$, 
	\item[ii(n)] $x_n$ is positive and $x_n=a_n+b_n\sqrt{y_n}\in K_{n+1}$ where $a_n\neq 0$ and $b_n$ are in $K_n$, and $|b_n|\geq 1$,
	\item[iii(n)] $y_n>x_n^2$.  
\end{enumerate}
If $n=0$, i(n), ii(n) and iii(n) are just the assumptions of the lemma. 

Assume i(n), ii(n) and iii(n) hold. Since $y_n>x_n^2$ there is an integer $k_{n+1}\geq 1$ such that $y_{n+1}=y_n-k_{n+1}x_n^2$ with $0\leq y_{n+1}<x_n$. Now by ii(n), 
\begin{align*}
y_{n+1}&=y_n-k_{n+1}(a_n^2+b_n^2y_n+2a_nb_n\sqrt{y_n})\\
&=-(k_{n+1}a_n^2+(k_{n+1}b_n^2-1)y_n)-2a_nb_nk_{n+1}\sqrt{y_n}\\
&=-z_n-2a_nb_nk_{n+1}\sqrt{y_n}
\end{align*}  
with $z_n>0$ and $a_nb_nk_{n+1}\neq 0$. By the previous Lemma, $y_{n+1}$ is not a square in $K_{n+1}=K(\sqrt{y_n})$. In particular $y_{n+1}$ is not zero and satisfies $0<y_{n+1}<x_n^2$. Therefore i(n+1) holds. 

Since $0<y_{n+1}<x_n^2$, there exists an integer $j_{n+1}\geq 1$ such that $x_{n+1}=x_n-j_{n+1}\sqrt{y_{n+1}}$ with $0\leq x_{n+1}<\sqrt{y_{n+1}}$.

Notice that $x_{n+1}\neq 0$ for $y_{n+1}$ is not a square in $K_n(\sqrt{y_n})$.  
Therefore, $0<x_{n+1}^2<y_{n+1}$, and iii(n+1) holds. Beside, $x_{n+1}=x_n-j_{n+1}\sqrt{y_{n+1}}$ where $x_n$ is a positive element  in $K_n(\sqrt{y_n})$,
hence $x_{n+1}=a_{n+1}+b_{n+1}\sqrt{y_{n+1}}\in K_n(\sqrt{y_n})(\sqrt{y_{n+1}})$ with $|b_{n+1}|\geq 1$ which in turn implies that ii(n+1) holds.
\end{proof}
\begin{proof}[End of proof of Theorem \ref{thm:density2}]
By the above Lemma, the orbit of point the  $M_0=(-1+\sqrt{2},2)$ under the action of $\Gamma(\sigma)$ is dense in $\R^2$. Now,
\[
h_{\sigma}\circ v_{\sigma}^2\circ h_{\sigma}^{-1}\circ v_{\sigma}^4(1,0)=M_0,
\]
hence the orbit of $(1,0)$ is dense. Next  the homeormorphism $\phi_t(x,y)=(tx,\sgn(t)t^2y)$ commute with $h_{\sigma}$ and $v_{\sigma}$, therefore $\Gamma(\sigma)(t,0)=\phi_t(\Gamma(\sigma)(1,0))$ is dense in $\R^2$ for all $t\in\R$. 
It follows that the orbit of any nonzero point in any rational line is dense. Since the orbit of the other points are also dense, we are done.
\end{proof}
\begin{remark}\label{rem:rational2}
The point $M_0$ is in the image of the $Ox$ axis by $g=h\circ v^2\circ h^{-1}\circ v^4$ which belongs to $\Gamma(\sigma)$. But $M_0$ is not in a $\sigma$-rational line, hence it provides an example where the set of rational lines we defined, is strictly included in the set of lines $g(Ox)$, $g\in\Gamma(\sigma)$ (see remark \ref{rem:rational}).
\begin{remark}
Suppose that $x$ and $y$ are two real numbers algebraically dependent over $\Q$. Then, it is not difficult to see that  $h_{\sigma}(x,y)$, $v_{\sigma}(x,y)$, $h_{\sigma}^{-1}(x,y)$ and $v_{\sigma}^{-1}(x,y)$ are still  algebraically dependent pairs over $\Q$. It follows that all the $(x,y)$ in lines the $g(Ox)$, $g\in \Gamma(\sigma)$, are algebraically dependent. This gives plenty of examples of $(x,y)$ that are not in a $\sigma$-rational line, for instance $(1,\pi)$.
\end{remark}
\end{remark}
\section{An abstract lemma about extensions of invariant measure}\label{sec:abstract}
We shall use a local formulation of Nogueira's Theorem. This local formulation can be derived from Nogueira's Theorem either by a careful examination of its proof or by using some general facts about extensions of invariant measures. Thereby, the Lemma below will allow to obtain this local formulation. 
\begin{definition} Let $(X,\mathcal A)$ be a measurable space and let $E$ be set of measurable and one to one maps from $X$ to $X$ such that for all $A\in \mathcal A$ and all $g\in E$, $gA\in\mathcal A$ (the maps $g$ are {\it bi-measurable}). A measure $\mu$ defined on the measurable subsets of a subset $B\in\mathcal A$ is $E$-invariant on $B$ if for all measurable $A\subset B$ and all $g\in E$, 
	\[
	gA\subset B\Rightarrow \mu(gA)=\mu(A).
	\]
\end{definition}

\begin{lemma}\label{lem:extension}
	Let $(X,\mathcal A)$ be a measurable space and let $h,v:X\fff X$ be two bi-measurable one to one maps such that $h(X)$ and $v(X)$ form a partition of $X$. Let $B$ be a measurable subset in $X$ such that $X$ is the union of all the $wB$, when $w$ ranges in the monoid $\mathcal M(h,v)$. 
	\begin{enumerate}
		\item If $\mu$ and $\nu$ are two $\{h,v\}$-invariant measures on $X$ that are equal on $B$ then $\mu=\nu$ on $X$.
		\item Any $\mathcal M(h,v)$-invariant measure on $B$ is the restriction of an $\mathcal M(h,v)$-invariant measure on $X$.  
		\item Suppose that $h(X\setminus B)\cap B=v(X\setminus B)\cap B=\emptyset$. Then any $\{h,v\}$-invariant measure on $B$ is the restriction of an $\mathcal M(h,v)$-invariant measure on $X$.
	\end{enumerate}
\end{lemma} 
\begin{remark}
The monoid is assumed to contain the identity map.
\end{remark}
\begin{proof}
	{\sc Preliminary observations.}\\
	Since $h(X)$ and $v(X)$ are disjoint the monoid $\mathcal M(h,v)$ is free and we consider elements in this monoid as words in $h$ and $v$. The empty word corresponds to the identity map.
	
	If $w$ and $w'$ are two words in the monoid, the intersection  $wX\cap w'X$ is non empty iff $w$ is a prefix of $w'$ or $w'$ is a prefix of $w$. In the first case $w'X\subset wX$ and $wX\subset w'X$ in the second case.
	
	For $w\in \mathcal M(h,v)$, set 
	\begin{align*}
	&X_w=wB\setminus \bigcup_{w' \text{ prefix of }w}w'B \\ &\text{and }
	B_w=w^{-1}X_w.
	\end{align*}
	The sets $X_w$, $w\in\mathcal M(h,v)$, form a partition of $X$. Indeed their union is $X$ by assumption, and by the above remark about prefix, they are pairwise disjoint.
	
	1. If $\mu$ and $\nu$ are two an $\{h,v\}$-invariant measures equal on $B$, since for all  $A\in\mathcal A$ and all $w\in\mathcal M(h,v)$, $w(w^{-1}(A\cap X_w))=(A\cap X_w)$, we have
	\begin{align*}
	\mu(A)&=\sum_{w\in\mathcal M(h,v)}\mu(A\cap X_w)\\
	&=\sum_{w\in\mathcal M(h,v)}\mu(w^{-1}(A\cap X_w))\\
	&=\sum_{w\in\mathcal M(h,v)}\nu(w^{-1}(A\cap X_w))\\
	&=\sum_{w\in\mathcal M(h,v)}\nu(A\cap X_w)=\nu(A).
	\end{align*} 
	
	2. Suppose now that $\mu$ is  a $\mathcal M(h,v)$-invariant measure on $B$ and define $\nu$ on $X$ by 
	\[
	\nu(A)=\sum_{w\in\mathcal M(h,v)}\mu(w^{-1}(A\cap X_w))
	\]
	for $A\in\mathcal A$. Clearly $\nu$ is a measure which coincide with $\mu$ for all $A\subset B$. Therefore to prove that $\nu$ is an invariant extension to $X$ of $\mu$, the only thing to check is the invariance:\\
	If $g=h$ or $v$ and if $A$ is a measurable subset of $X_w$ for some $w\in \mathcal M(h,v)$ then $\nu(gA)=\nu(A)$.  	
	
	Since $g$ is one to one,
	\begin{align*}
	gX_w&=gwB\setminus \bigcup_{w' \text{ prefix of }w}gw'B,\\
	&=gwB\setminus \bigcup_{w' \text{ prefix of }gw,\,w'\neq Id}w'B\\
	&=X_{gw}\cup B'
	\end{align*} 
	where $B'= gX_w\setminus X_{gw}$ is included in $B$. Since $gA\subset gX_w$, we have a partition $A=A'\cup A''$ with $gA'=B'\cap gA $ and $gA''=X_{gw}\cap gA$.	By definition of $\nu$,
	\begin{align*}
	\nu(gA)&=\nu(gA\cap X_{gw})+\nu(gA\cap B')\\
	&=\mu((gw)^{-1}(gA''))+\mu(gA')\\
	&=\mu(w^{-1}A'')+\mu(gw(w^{-1}A')).
	\end{align*}
	Since $w^{-1}A'\subset B$ and $gw(w^{-1}A')\subset B$, the invariance of $\mu$ on $B$ implies that $\mu(gw(w^{-1}A'))=\mu(w^{-1}A')$. Therefore 
	\begin{align*}
	\nu(gA)=\mu(w^{-1}A'')+\mu(w^{-1}A')
	=\mu(w^{-1}A)=\nu(A).
	\end{align*}
	3. Under the assumption $h(X\setminus B)\cap B=v(X\setminus B)\cap B=\emptyset$, the $\{h,v\}$-invariance on $B$ implies the $\mathcal M(u,v)$-invariance on $B$. So (3) is a consequence of (2).
\end{proof}
	
\section{Nogueira's theorem and its consequences}\label{sec:nogueira}
All the measures we consider are defined on Borel subsets of $\R^2$ and we
denote $\lambda$  the Lebesgue measure on $\R^2$.

For $R>0$, let $\mathcal I_{\sigma}(R)$ be the set of positive  measures $\mu$ on the square $S(R)=[0,R]^2$ such that
\begin{itemize}
\item $\mu(S(R))$ is  finite,
\item $\mu(\text{coordinate axes})=0$,
\item $\mu$ is $\mathcal M(h_{\sigma},v_{\sigma})$-invariant on $S(R)$,
\end{itemize}
When $\sigma$ is the identity we drop the index $\sigma$ and simply write $\mathcal I(R)$.
Let $\mathcal I^1(R)$ denote the set of $\mu\in \mathcal I(R)$ such that $\mu(S(R))=1$.

\begin{theorem}[Nogueira] 
	Up to a multiplicative constant, there is only one   $\{h,v\}$-invariant measure on $\R_{\geq 0}^2$ that does not charge the axes.
\end{theorem}

\begin{corollary} \label{cor:nogueira}
For all positive real number $\rho$, there is only one normalized measure on $S(\rho)=[0,\rho]^2$ that is $\{h,v\}$-invariant on $S(\rho)$, and that does not charge the axes, i.e., $\mathcal I^1(\rho)$ has only one element.
\end{corollary}
\begin{proof}[Proof of the corollary] Let $Z$ be the set of points in $\R^2_{\geq 0}$ that are not
in a rational line. By Lemma \ref{lem:1/2}, any point in $Z$ can be map by a positive iterate of the Euclidean inside $S(\rho)$, hence $Z\subset\bigcup_{g\in\mathcal{M}(h,v)}gS(\rho)$. 
Therefore by Lemma \ref{lem:extension}, any $\{h,v\}$-invariant measure on $S(\rho)$ with total mass $1$, that does not charge the rational lines can be extended to $Z$. 
Nogueira's Theorem implies that this measure is proportional to the Lebesgue measure and so is its restriction to $S(\rho)$. 
\end{proof}

\subsection{Invariant subsets and proof of Theorem \ref{thm:uniqueness}}
We begin by to general facts which hold for any increasing odd homeomorphism $\sigma$.
\begin{lemma}
	If $A$ is a subset of $\R^2$  invariant by $h_{\sigma}$ and $v_{\sigma}$ then $A_+=A\cap\R^2_{\geq 0}$ is invariant by the generalized Euclide algorithm, i.e., $E_{\sigma}^{-1}(A_+)=A_+$.	
\end{lemma}
\begin{proof}
	On the one hand,
	\begin{align*}
	E_{\sigma}^{-1}(A_+)\subset h_{\sigma}(A_+)\cup v_{\sigma}(A_+)\subset A\cap\R^2_{\geq 0}=A_+.
	\end{align*}
	On the other hand, consider  $(x,y)\in A_+$. If $ \sigma(x)>y $, then $h_{\sigma}^{-1}(x,y)\in A_+$ and $E_{\sigma}(x,y)=h_{\sigma}^{-1}(x,y)$. Therefore $(x,y)\in E_{\sigma}^{-1}(A_+)$. If $\sigma(x)\leq y$  then $v_{\sigma}^{-1}(x,y)\in A_+$ and $E_{\sigma}(x,y)=v_{\sigma}^{-1}(x,y)$. Therefore $(x,y)\in E_{\sigma}^{-1}(A_+)$.   
\end{proof}
\begin{lemma}\label{lem:invariant}
	Let $A$ and $S$ be Borel subsets in $\R^2_{\geq 0}$. Suppose that $A$ is invariant by the generalized Euclide algorithm. Then for all measure $\mu$ on $S$, $\{h_{\sigma},v_{\sigma}\}$-invariant on $S$, that does not charge the axes, the measure $\nu=1_{A\cap S}\mu$ is also $\{h_{\sigma},v_{\sigma}\}$-invariant on $S$.
\end{lemma}
\begin{proof}
	Let $B$ be a Borel subset of $S$ such that $h_{\sigma}(B)\subset S$. We want to prove that $\nu(h_{\sigma}(B))=\nu(B)$. It is enough to prove that $A\cap h_{\sigma}(B)=h_{\sigma}(A\cap B) \mod \mu$. 
	
	Let $(x,y)\in A\cap B$. Since $(x,y)\in A=E_{\sigma}^{-1}(A)$, $E_{\sigma}(x,y)\in A$. If $x>0$, then $E_{\sigma}(h_{\sigma}(x,y))=(x,y)$. Therefore $h_{\sigma}(x,y)\in E_{\sigma}^{-1}(A)=A$. 
	It follows that $h_{\sigma}(A\cap B)\subset (A\cap h_{\sigma}(B))\cup h_{\sigma}(\{0\}\times \R_{\geq 0})$.
	
	Conversely, if $u=h_{\sigma}(x,y)\in A\cap h_{\sigma}(B)$ with $(x,y)\in B$ and  $x>0$ then $(x,y)=E_{\sigma}(h_{\sigma}(x,y))=E_{\sigma}(u)\in A$ because $u\in A=E_{\sigma}^{-1}(A)$. Therefore $(x,y)\in A\cap B$ which implies that $u\in h_{\sigma}(A\cap B)$. It follows $(A\cap h_{\sigma}(B))\setminus h_{\sigma}(\{0\}\times \R_{\geq 0})$ is included in $h_{\sigma}(A\cap B)$.
	
	The same proof can be done with $v_{\sigma}$ instead of $h_{\sigma}$. 
\end{proof}

\begin{proof}[Proof of Theorem \ref{thm:uniqueness}]
Using the change of variables $\phi(x,y)=(ax,y)$ we can reduce to $a=1$. Indeed, 
\begin{align*}
\phi\circ h_{\sigma}\circ\phi^{-1}(x,y)
=\phi(\tfrac{1}{a}x+\sigma^{-1}(y),y)=(x+a\sigma^{-1}(y),y)
\end{align*}
which is equal to $(x+y,y)$ when $y\in[-ab,ab]$. And $\phi\circ v_{\sigma}\circ\phi^{-1}(x,y)=(x,x+y)$ for $x\in[-ab,ab]$ as well. 

Suppose that $a=1$ and let $\mu$ be a measure in $\mathcal I_{\sigma} (\R^2_{\geq 0})$.  Since $a=1$, the restriction of $\mu$ to the square $S(b)$ is in $\mathcal I(S(b))$. Therefore by the corollary of Nogueira's theorem, $\mu=c\lambda$ on $S(b)$ for some nonnegative constant $c$. Now $\mu$ and $\lambda$ are in $\mathcal I_{\sigma}(\R^2_{\geq 0})$, hence by the uniqueness part of Lemma \ref{lem:extension}, $\mu=c\lambda$ on $\R^2_{\geq 0}$. 
\end{proof}
\begin{proof}[Proof of Corollary \ref{cor:ergodic}]
Let $A\subset\R^2_{\geq 0}$ be an $E_{\sigma}$ invariant  Borel set. By Lemma \ref{lem:invariant}, the measure $\nu=1_{A}\lambda$ is in $\mathcal I_{\sigma}(\R^2_{\geq 0})$. Theorem \ref{thm:uniqueness} implies that $\nu$ is proportional to the Lebesgue measure. 
Which in turn implies $\lambda(A)=0$ or $\lambda(\R^2_{\geq 0})\setminus A)=0$.
\end{proof}

\begin{proof}[Proof of Corollary \ref{cor:exact}]
As above, using the change of variables $\phi(x,y)=(ax,y)$, we can reduce to $a=1$.
Since the Euclidean algorithm is ergodic, it is exact if and only if  for any Borel set $B$ in $\R^2_{\geq 0}$ of positive Lebesgue measure, there exists a positive integer $k$ such that $\lambda(E_{\sigma}^k(B)\cap E_{\sigma}^{k+1}(B))>0$, see Lemma 2.1 of \cite{MiNo}. 

Let $B$ be a Borel subset in $\R^2_{\geq 0}$ of positive Lebesgue measure. By Lemma $\ref{lem:1/2}$, there exist a positive integer $n$ and $A\subset B$ measurable and of positive measure such that $C=E_{\sigma}^n(A)\subset S(b)$. Restricted to $S(b)$, $E_{\sigma}=E$ which is exact. Therefore, there exists a positive integer $k$ such that $\lambda(E^k(C)\cap E^{k+1}(C))>0$. It follows that $E_{\sigma}^{n+k}(B)\cap E_{\sigma}^{n+k+1}(B)$ contains $E^k(C)\cap E^{k+1}(C)$ and is of positive Lebesgue measure.
\end{proof}

\section{$\sigma$-lines and the generalized Euclidean algorithm}\label{sec:lines}
Let $\Omega_{+}=\Omega_+(\sigma)$ be the set of all infinite words written with the two ``letters'' $h_{\sigma}$ and $v_{\sigma}$
\begin{proposition}
1. For all $w=(w_n)_{n\in\N}$ in $\Omega_+$ not equal to $h_{\sigma}^{\infty}$ or to $v_{\sigma}^{\infty}$, the set 
\[
L_w=\bigcap_{n\in\N}w_0\circ w_2\circ\dots\circ w_n(\R_{\geq 0}^2)
\]
is the graph of an increasing homeomorphism $\sigma_w:\R_{\geq 0}\fff\R_{\geq 0}$.\\
2. Every point $p$ in $\R_{> 0}^2$ is contained in a $\sigma$-line $L_w$ for some $w\in \Omega_+(\sigma)$. Furthermore, $w$ is unique if and only if $p$ is not in a $\sigma$-rational line.\\ 
3. If $p$ is in a $\sigma$-rational line, there are exactly two elements $w$ and $w'$ in $\Omega_+$ such that $p\in L_w=L_{w'}$.
Furthermore, there exists a finite length word $m$ such that
\[
w=mh_{\sigma}v_{\sigma}^{\infty}, \text{ and }
w'=mv_{\sigma}h_{\sigma}^{\infty}.
\]
\end{proposition}

\begin{proof}
We only prove the first assertion of the proposition which is most the difficult part. Let $w=(w_n)_{n\in\N}$ be in $\Omega_+$ and not equal to $h_{\sigma}^{\infty}$ or to $v_{\sigma}^{\infty}$.
For all positive $t$,
\[
L_w\cap((t,0)+\R(0,1)) =\bigcap_{n\in\N}(w_0\circ w_1\circ\dots\circ w_n(\R_{\geq 0}^2)\cap ((t,0)+\R(0,1)))
\]
is the decreasing intersection of  the connected sets $w_0\circ w_1\circ\dots\circ w_n(\R_{\geq 0}^2)\cap ((t,0)+\R(0,1))$ which are compact provided that $w_i=h_{\sigma}$ for at least one $i\leq n$. Therefore $L_w\cap((t,0)+\R(0,1))$ is a nonempty compact  connected set. This holds for $L_w\cap((0,t)+\R(1,0))$ as well.

Let us show that the sets $L_w\cap((t,0)+\R(0,1))$  reduce to one point for all positive $t$.
Suppose on the contrary that there exists a positive $t$ and $w$ in $\Omega_+(\sigma)$ not constant such that $L_w\cap((t,0)+\R(0,1)))=\{t\}\times [a,b]$ where $0\leq a<b$. Since $w$ is not $h_{\sigma}^{\infty}$ or $v_{\sigma}^{\infty}$, the word $w$  begin either by $h_{\sigma}^{n_0}v_{\sigma}$ or by  
$v_{\sigma}^{n_0}h_{\sigma}$ for some $n_0>0$. 

If the  first letter $w_0$ is $h_{\sigma}$ then $h_{\sigma}^{-1}(L_w)=L_{w'}$ where $w'$ is obtained from $w$ by deleting the first letter. So $L_{w'}$ contains the points $(t-\sigma^{-1}(a),a)$ and $(t-\sigma^{-1}(b),b) $. It follows that the nonempty open rectangle 
\[
R=]t-\sigma^{-1}(b),t-\sigma^{-1}(a)[\times ]a,b[
\]
is included  in the interior of the sets $w_1\circ w_2\circ\dots\circ w_n(\R_{\geq 0}^2)$ for all $n\geq 1$. The boundary of $w_1\circ w_2\circ\dots\circ w_n(\R_{\geq 0}^2)$ is the union of the two $\sigma$-rational lines $w_1\circ w_2\circ\dots\circ w_n(Ox)$ and $w_1\circ w_2\circ\dots\circ w_n(Oy)$. So none of these two $\sigma$-rational lines meets $R$. 

Let $L=u_1\circ u_2\circ\dots\circ u_n(Ox)$ be a $\sigma$-rational line. Suppose that $u_i=w_i$ for $i<j\leq n$ and $u_j\neq w_j$. If $u_j=h_{\sigma}$ and $w_j=v_{\sigma}$ then  $u_1\circ u_2\circ\dots\circ u_n(Ox)$ is under $u_1\circ u_2\circ\dots\circ u_j(Oy)=w_1\circ w_2\circ\dots\circ w_j(Ox)$ which in turn is under $w_1\circ w_2\circ\dots\circ w_n(Ox)$ and so $L$ does not meet $R$. 
If $u_j=v_{\sigma}$ and $w_j=h_{\sigma}$ then $u_1\circ u_2\circ\dots\circ u_n(Ox)$ is above $u_1\circ u_2\circ\dots\circ u_j(Ox)=w_1\circ w_2\circ\dots\circ w_j(Oy)$ which in turn is above $w_1\circ w_2\circ\dots\circ w_n(Oy)$ and so $L$ does not meet $R$.

 Using the same way of reasoning with $\sigma$-rational lines $L=u_1\circ u_2\circ\dots\circ u_n(Oy)$, we see that all $\sigma$-rational lines have  empty intersections with $R$ which contradicts the density of the $\sigma$-rational lines (see Proposition \ref{prop:guido}). So the sets $L_w\cap((t,0)+\R(0,1))$  reduce to one point for all positive $t$ when the first letter of $w$ is $h_{\sigma}$.

If the first letter of $w$ is $v_{\sigma}$ then we can apply the previous case to the word $v_{\sigma}^{-n_0}w$ because $L_{v_{\sigma}^{-n_0}w}\cap((t,0)+\R(0,1))) $ still contains a vertical segment of length $b-a$. 

So in all cases we have shown that $L_w$ cannot contain nontrivial vertical segment. The proof that $L_w$ does not contains nontrivial horizontal segment goes the same way, just reverse the role of $h_{\sigma}$ and $v_{\sigma}$.

The fact that $L_w$ does not contains neither vertical segments nor horizontal segments implies that $L_w$ is the graph of a bijection $\sigma_w:\R_{\geq 0}\fff\R_{\geq 0}$. To prove that $\sigma_w$ is an homeomorphism it is enough to see that it is nondecreasing. 

Let $w$ be an infinite non constant word. For each $n$, $w_0\circ w_1\circ\dots\circ w_n(Ox)$ is the graph of an increasing function $f_n:\R_{\geq 0}\fff\R_{\geq 0}$. Furthermore, for each positive $t$, $f_n(t)$ is the infimum of $(w_0\circ w_1\circ\dots\circ w_n(\R_{\geq 0}^2)\cap ((t,0)+\R(0,1)))$.  Since
\[
\{(t,\sigma_w(t)\}=
L_w\cap((t,0)+\R(0,1)) =\bigcap_{n\in\N}(w_0\circ w_1\circ\dots\circ w_n(\R_{\geq 0}^2)\cap ((t,0)+\R(0,1)),
\]
we have $\sigma_w(t)=\lim_{n\fff\infty}f_n(t)$. It follows that $\sigma_w$ is the point-wise limit of a sequence of nondecreasing function and hence is nondecreasing.
\end{proof}

Call {\it $\sigma$-line}, the sets $L_w$, $w\in\Omega_{+}$. Observe that $Ox=L_{h_{\sigma}^{\infty}}$ and $Oy=L_{v_{\sigma}^{\infty}}$
A change of variable using the ${\sigma}$-lines  can be used to express  the action of the generalized Euclidean algorithm. 
These new variables should help to find invariant measures for the generalized Euclidean algorithm, at least in some particular cases.\\  
We need a few notations.
\begin{itemize}
	\item  Let  $s:\Omega_+\fff\Omega_+$ denote the shift map defined by $s(w_0w_1w_2\dots)=w_1w_2\dots$.
	\item Let $\phi: \R_{> 0}\times(\Omega_+\setminus\{v_{\sigma}^{\infty}\})\fff\R_{> 0}\times \R_{>0}$ denote  the ``change of variable'' defined by
	\[
	\phi(x,w)=p \text{ where } \{p\}=L_w\cap((x,0)+\R(0,1)).
	\] 
	\item Let  $H_{\sigma}:\R_{> 0}\times(\Omega_+\setminus\{v_{\sigma}^{\infty}\})\fff \R_{\geq 0}\times \Omega_+$ denote the map defined by
	\[
	H_{\sigma}(x,w)=
	\left\{
	\begin{array}[l]{lll}
	(x-\sigma^{-1}\circ\sigma_w(x),s(w)) &\text{ if }& w_0=h_{\sigma}\\
	(x,s(w)) &\text{ if }& w_0=v_{\sigma}\\
	\end{array}
	\right..
	\]
	\item Let  $Z$ denote the set of points in $\R^2_{\geq 0}$ that are not
	in a $\sigma$-rational line.
	\item Let  $\Omega_{+}^0$ denote the set of words in $\Omega_+$ whose tail are not of the form $h_{\sigma}^{\infty}$ or $v_{\sigma}^{\infty}$.
\end{itemize}
\begin{remark}
Observe that for any $w\in\Omega_+^0$ and any $x>0$,
\[
	\left\{
	\begin{array}[l]{lll}
	\sigma^{-1}_{s(w)}\circ \sigma_w(x)=x-\sigma^{-1}\circ\sigma_w(x) &\text{ if }& w_0=h_{\sigma}\\
	\sigma_{s(w)}(x)=\sigma_{w}(x)-\sigma(x) &\text{ if }& w_0=v_{\sigma}\\
	\end{array}
	\right..
	\]
\end{remark}
Standard arguments lead to
\begin{proposition}
Endow $\Omega_+^0$ with  the $\sigma$-algebra induced by the product $\sigma$-algebra on $\Omega_+$,   and then endow $\R_{> 0}\times\Omega_+^0$ with the product $\sigma$-algebra. We have
\begin{enumerate}
\item $\phi:\R_{> 0}\times\Omega_+^0\fff Z$ is bijective and bi-measurable.
\item The measurable dynamical systems $(Z,E_{\sigma})$ and $(\Omega_+^0\times \R_{> 0},H_{\sigma})$ are isomorphic:   
\[
\phi \circ H_{\sigma}(x,w)=E_{\sigma}\circ \phi(x,w)
\]
for all $(x,w)\in\R_{> 0}\times\Omega_+^0$.
\end{enumerate} 
\end{proposition}

\section{About invariant measures for the generalized Euclidean algorithm} \label{sec:invariant}

The shift map $s:\Omega_+\fff\Omega_+$ has plenty of invariant measures. Choose one such measure $\mu$ with $\mu(\Omega_+\setminus\Omega_{+}^0)=0$. This measure does not need  to be of finite total mass. Let $\mathcal B(\R_{>0})$ be the Borel $\sigma$ algebra of $\R_{> 0}$. Consider a kernel
\begin{align*}
K:\mathcal B(\R_{>0})\times\Omega_+\fff\R_{\geq 0}\cup\{\infty\}
\end{align*}
such that
\begin{itemize}
\item for all $w\in\Omega_{+}$, $K(.,w)$ is a positive measure on $\R_{> 0}$,
\item for all $B\in\mathcal B(\R_{> 0})$, $K(B,.)$ is measurable.
\end{itemize}
If for $\mu$-almost all $w$, the image of the measure $K(.,w)$ by the map
\[
\varphi_w:x\in\R_{>0}\fff
\left\{
	\begin{array}[l]{lll}
	x-\sigma^{-1}\circ\sigma_w(x) &\text{ if }& w_0=h_{\sigma}\\
	x &\text{ if }& w_0=v_{\sigma}\\
	\end{array}
	\right.
\]
is the measure $K(s(w),.)$, then the measure $\nu$ defined on $ \R_{> 0}\times\Omega_{+}^0$ by
\[
\nu(A\times B)=\int_{B}d\mu(w)\int_{A}K(dx,w)
\]
is $H_{\sigma}$-invariant. Indeed, in that case we have
\begin{align*}
\nu(H_{\sigma}^{-1}(A\times B))&=\int d\mu(w)1_B\circ s(w)\int K(dx,w)1_A\circ\varphi_w(x)\\
&=\int d\mu(w)1_B\circ s(w)\int K(dx,s(w))1_A(x)\\
&=\int d\mu(w)1_B(w)\int K(dx,w)1_A(x)\\
&=\nu(A\times B).
\end{align*}
For a general homeomorphism $\sigma$,  the construct of such a kernel is not clear. When $\sigma(x)=\sgn(x)|x|^{\alpha}$ with $\alpha>0$, the kernel $K$ defined by
\[
K(A,w)=\int_A\frac1x \,dx
\]
satisfies the required invariance. This is a direct consequence of the following lemma.
\begin{lemma}
Suppose that $\sigma(x)=\sgn(x)|x|^{\alpha}$ for some $\alpha>0$. Then for all $w\in \Omega_{+}$, there exist a positive real number $k(w)$ such that $\varphi_w(x)=k(w)x$ for all $w\in \Omega_+\setminus\{v_{\sigma}^{\infty}\}$ and all $x>0$.
\end{lemma}
\begin{proof}
Consider the family of curves $c_a:x\in \R_{> 0}\fff(x,a\sigma(x))$ with $a\geq 0$.
For any curve $c_a$, $h_{\sigma}\circ c_a(x)=c_b((1+a^{1/\alpha})x)$ with
$b=\frac{a}{(1+a^{1/\alpha})^{\alpha}}$ and $v_{\sigma}\circ c_a(x)=c_b(x)$ with $b=a+1$.
Using the previous proposition we see that for each $w\in\Omega_+\setminus\{v_{\sigma}^{\infty}\}$ the boundary of $w_0\circ w_2\circ\dots\circ w_n(\R_{\geq 0}^2)$ is the union of two curves $c_{a_n}$ and $c_{b_n}$ where $(a_n)_n$ and $(b_n)_n$ are two adjacent sequences converging to a same limit $a(w)$. Therefore 
$L_w=\{c_{a(w)}(x):x\geq 0\}$. 

If $w_0=h_{\sigma}$ then $a(w)<1$ and 
\[
E_{\sigma}(c_{a(w)}(x))=h_{\sigma}^{-1}(c_{a(w)}(x))=c_b((1-a(w)^{1/\alpha}))x)
\]
with $b=\frac{a(w)}{(1-a(w)^{1/\alpha})^{\alpha}}$.

And if $w_0=v_{\sigma}$ then $a(w)\geq 1$ and
\[
E_{\sigma}(c_{a(w)}(x))=v_{\sigma}^{-1}(c_{a(w)}(x))=c_b(x)
\]
 with $b=a(w)-1$. This implies that
 \[
 H_{\sigma}(x,w)=(k(w)x,s(w))\, \text{ with } \,
 k(w)=\left\{
  \begin{array}[l]{lll}
  1-a(w)^{1/\alpha}&\text{ if }& w_0=h_{\sigma}\\
  1&\text{ if }& w_0=v_{\sigma}
  \end{array}
  \right..
 \]
\end{proof}
Thanks to the above lemma, we see that each $s$-invariant measure gives rise to a measure invariant by the generalized Euclidean algorithm:
\begin{corollary}
Suppose that $\sigma(x)=\sgn(x)|x|^{\alpha}$ for some $\alpha>0$. Let $\mu$ be any $s$-invariant measure on $\Omega_{+}^0$. Then the measure $\nu$   defined on the product $\R_{> 0}\times\Omega_+^0$ by
\[
\nu(A\times B)=\int_{B}d\mu(w)\int_{A}\frac{dx}{x}
\]
is $H_{\sigma}$-invariant.
\end{corollary}

\section{ The affine group associated with $\Gamma(\sigma)$ when $\sigma(x+1)=\sigma(x)+1$ }\label{sec:kazdhan}
In this section, we suppose that $\sigma(x+1)=\sigma(x)+1$ for all $x\in\R$.
Clearly,
\begin{align*}
&h_{\sigma}((x,y)+(n,m))=h_{\sigma}(x,y)+h(m,n),\\
&v_{\sigma}((x,y)+(m,n))=v_{\sigma}(x,y)+v(m,n)
\end{align*} 
for all $(m,n)\in
\mathbb{Z}^{2}$,  so that $\mathbb{Z}^{2}$ is stable by $\Gamma
(\sigma)$ and the restriction map  $\pi:g\rightarrow g|_{\mathbb{Z}^{2}}$
is a surjective morphism from $\Gamma(\sigma)$ onto  $\SL(2,\mathbb{Z)}$. This allows to extend $\Gamma(\sigma)$ with the $\Z^2$-translations: let $A\Gamma(\sigma)$ the group of all the maps of the form
\[p\in\R^2\fff\gamma(p)+u\] where $\gamma\in\Gamma(\sigma)$ and $u\in\Z^2$. The group $A\Gamma(\sigma)$ is the semi-direct product of $\Gamma(\sigma)$ with the additive group
 $(\mathbb{Z}^{2},+)$. It is known that $(A\Gamma(Id),\mathbb{Z}^{2})$ has the relative
Kazdhan property \cite{Sh}. By Corollary 2
in \cite{CoTe},  $(A\Gamma(\sigma),\mathbb{Z}^{2})$ has also the relative Kazdhan property.  A direct proof is also possible by adapting Shalom's proof.

When  $n\geq3$ there are similar result:   $\mathbb{Z}^{n}$ is stable under the action of
$\Gamma(n,\sigma)$, the restriction $\pi:g\rightarrow g|_{\mathbb{Z}^{n}}$ is a
surjective morphism from $\Gamma(n,\sigma)$ on $\SL(n,\mathbb{Z)}$. As above, let denote
 $A\Gamma(n,\sigma)$  the semi-direct product of $\Gamma(n,\sigma)$ with the additive group
 $(\mathbb{Z}^{n},+)$.  The couple $(A\Gamma(n,\sigma),\mathbb{Z}^{n})$ has the relative Kazdhan property.

\section{Invariant measures in the two-dimensional torus}\label{sec:torus}
The generalized transvections can be defined in the two-dimensional torus $\T^2$. Let $\sigma_1,\sigma_2:\T=\R/\Z\fff\T$ be two measurable maps and define $h,v:\T^2\fff\T^2$ by
\begin{align*}
h(x,y)&=(x+\sigma_2(y),y),\\
v(x,y)&=(x,y+\sigma_1(x)).
\end{align*} 
The two maps $h,v$ are invertible maps acting on $\T^2$ and preserve the Lebesgue measure $\lambda_{\T^2}$ of the two-dimensional torus.
Call $\Gamma=\Gamma(\sigma_1,\sigma_2)$ the group generated by 
$h$ and $v$. Some results about invariant measures are easy in this setting. Denote $\lambda_{\T^d}$ the normalized Lebesgue measure on the torus $\T^d$.
\begin{proposition}\label{prop:torus}
For $i=1,2$, let $N_i=\sigma_i^{-1}(\Q/\Z)$.
\begin{enumerate}
\item If $\lambda_{\T}(\sigma_i^{-1}(\Q/\Z))=0$, $i=1,2$, then  any finite measure $\mu$ that is $\Gamma$-invariant and absolutely continuous  
with respect to the Lebesgue measure $\lambda_{\T^2}$, is proportional to $\lambda_{\T^2}$.  
\item If the sets $N_i=\sigma_i^{-1}(\Q/\Z)$, $i=1,2$, are countable, then any finite measure $\mu$ that is $\Gamma$-invariant and has  no atom, is proportional to $\lambda_{\T^2}$.	
\end{enumerate}
\end{proposition} 
\begin{corollary}
If $\lambda_{\T}(\sigma_i^{-1}(\Q/\Z))=0$, $i=1,2$, then the action of $\Gamma$ on $\T^2$ is ergodic with respect to the Leesgue measure.  
\end{corollary}

\begin{proof}[Proof of the proposition] Let $\mu$ be a $\Gamma$-invariant measure with total mass $1$. If $\mu$ is absolutely continuous with respect to the Lebesgue measure and $\lambda_{\T}(\sigma_i^{-1}(\Q/\Z))=0$, $i=1,2$, then  $\mu(N_1\times\T)=\mu(\T\times N_2)=0$. The same result holds if $\mu$ has no atom and the sets $N_i$ are countable. Indeed, 
let $r$ be in $\T$ and consider the set $E_r=\{(r,y):y\in\T\setminus N_2\}$. Since $\sigma_2(y)$ is irrational for all $y\notin N_2$, the sets $h^n(E_r)$, $n\in \Z$ are pairwise  disjoint and have the same $\mu$-measure, hence they all have zero measure. Since $\mu$ has no atom $F_r=\{(r,y):y\in\T\}$ has also zero measure which implies that $N_1\times \T$ has zero measure. The same proof shows that $\T\times N_2$ has zero measure.

Let $\varphi_1,\varphi_2:\T\fff\R$ be two continuous functions. Making use successively of the Birkhoff Theorem, of the zero measure of $\T\times N_2$ and of the unique ergodicity of irrational translations, we obtain
\begin{align*}
\int_{\T^2} \varphi_1\otimes\varphi_2(x,y) d\mu(x,y)&=\int_{\T^2}\lim_{n\fff\infty}\frac1n \sum_{k=0}^{n-1}(\varphi_1\otimes\varphi_2)\circ h^k(x,y) d\mu(x,y)\\
&=\int_{\T^2\setminus\T\times N_2}\lim_{n\fff\infty}\frac1n \sum_{k=0}^{n-1}(\varphi_1\otimes\varphi_2)\circ h^k(x,y) d\mu(x,y)\\
&=\int_{\T}\varphi_1d\lambda\int_{\T^2\setminus\T\times N_2}1_{\T}(x)\otimes\varphi_2(y) d\mu(x,y)\\
&=\int_{\T}\varphi_1d\lambda_{\T^1}\int_{\T}\varphi_2 d\mu_2
\end{align*}
where $\mu_2$ is the image of $\mu$ by the projection $(x,y)\in\T^2\fff y\in\T$.
	
Reversing the role of $x$ and $y$ we obtain,
\begin{align*}
\int_{\T^2} \varphi_1\otimes\varphi_2(x,y) d\mu(x,y)
&=\int_{\T}\varphi_2d\lambda_{\T^1}\int_{\T}\varphi_1 d\mu_1
\end{align*}
where $\mu_1$ is the image of $\mu$ by the projection $(x,y)\in\T^2\fff x\in\T$. Using the functions $1\otimes\varphi_2$ and $\varphi_1\otimes 1$,  it follows that $\lambda=\mu_1=\mu_2$ and then that $\mu=\lambda_{\T^1}\otimes\lambda_{\T^1}=\lambda_{\T^2}$.
\end{proof}

\section{Miscellaneous questions and remarks}\label{sec:questions}
\subsection{Density}
Suppose that $\sigma:\R\fff\R$ is given by $\sigma(x)=\sgn(x)|x|^{\alpha}$ for a positive $\alpha\neq 1$. Is the orbit of any nonzero point under the action of $\Gamma(\sigma)$, dense in in the plan? Start with the case $\alpha=3$.  (Recall that $\Gamma(\sigma)$ is the group generated by the two generalized tarnsvections)
\subsection{Ergodicity and invariant measures}
Suppose that $\sigma:\R\fff\R$ is an increasing odd homeomorphism.\\
Does the action of $\Gamma(\sigma)$ on the plan equipped with the Lebesgue measure always ergodic?\\
Same question for the generalized Euclidean algorithm acting on the first quadrant? \\
What are the $h_{\sigma},v_{\sigma}$-invariant measures on the first quadrant?\\ 
In the authors opinion the differentiability at the origin and the cancellation of the derivative should play a role in these three questions.  

Find ``many'' invariant measures absolutely continuous with respect to the Lebesgue measure, for the generalized Euclidean algorithm.

\subsection{Arithmetic result for the generalized transvections}
The proof of Nogueira's Theorem involved the following standard arithmetical result:
\[
\lim_{n\rightarrow\infty}\frac{1}{n^2}\operatorname{card}\{(p,q)\in\N^2:p,q\leq n \text{ and } \gcd(p,q)=1\}=\frac{6}{\pi^2}
\]
(Mertens Theorem \cite{Me}, Wikipedia: Cesàro Theorem).
Numerical simulations indicate that this Theorem could still hold in the context of generalized transvections. More precisely:
For $r>0$, consider the set $\Lambda(r)=\{w(r,r)\in [0,1]^2: w\in \mathcal M(h_{\sigma},v_{\sigma})\}$ (recall that $\mathcal M(h_{\sigma},v_{\sigma})$ is the monoid generated by $h_{\sigma}$ and $v_{\sigma}$). Study the limit 
\[
\lim_{r\fff 0}r^2\operatorname {card}\Lambda(r)
\]
Numerical simulations for $\sigma$ tangent to the identity at the origin, give a limit equal to $\frac{6}{\pi^2}$ as in Mertens' Theorem. For $\sigma(x)=\sgn(x)|x|^{\alpha}$, the limit seems to exist and be positive.  

Guido Ahumada et Nicolas Chevallier\\
Université de Haute Alsace\\
4, rue des frères Lumière,\\
68093 Mulhouse France


\begin{thebibliography}{99}

\bibitem{Bu}
M. Burger,
Kazhdan constants for $\SL_3(\Z)$,
{\sl J. Reine Angew. Math. }
{\bf 413}, (1991), 36--67.

\bibitem{CoTe}
Y. Cournulier, R. Tessera, A characterization of relative Kazdhan property T for semi-direct products with abelian groups. 
{\sl Ergod. Th. \& Dynam. Sys.}
{\bf 31} (2011), 793--805

\bibitem{Da1}
S. G. Dani, On invariant measure, minimal sets and a lemma of Margulis,
{\sl Invent. Math }
{\bf 51 } (1979), 239--260.


\bibitem{DaNo1}
S. G. Dani, A. Nogueira,
On invariant measures of the Euclidean algorithm
{\sl Ergod. Th. \& Dynam. Sys.}
{\bf 27} no 2, (2007), 417--425.

\bibitem{DaNo2}
S. G. Dani, A. Nogueira,
On $\SL(n,\Z)_+$-orbits on $\R^n$ and positive integral solutions of linear inequalities.
{\sl J. Number Theory}
{\bf 129} no 10, (2009) 2526--2529,

\bibitem{Me}
Mertens, Ueber einige asymptotische Gesetze der Zahlentheorie.
{\sl Journal für Math}
{\bf 77} (1874), 289--338.

\bibitem{MiNo}
T. Miernowski, A. Nogueira,
Exactness of the Euclidean algorithm and of the Rauzy induction on the space of interval exchange transformations,
{\sl Ergod. Th. \& Dynam. Sys.}
{\bf 33} no 1 (2013), 221--246.


\bibitem{No1}
A. Nogueira, The 3-dimensional Poincaré continued fraction algorithm,
{\sl Israel J. Math.}
{\bf 90} (1995) 373--401,

\bibitem{No2}
A. Nogueira,
Relatively prime numbers and invariant measures under the natural action of $\SL(n,\R)$ on $\R^n$,
{\sl Ergod. Th. \& Dynam. Sys.}
{\bf 22} (2002), 899--923.

\bibitem{Se} J.P. Serre,
{\sl Arbres, Amalgames, $SL_2$},
Astérisque, S.M.F., 1977.


\bibitem{Sh}
Y. Shalom, Bounded generation and Kazhdan's propoerty (T),
{\sl Inst. Hautes \'Etudes Sci. Publ. Math. }
{\bf 90} (1999), 145--168.


\bibitem{Wa}
S. Wagon,
The Banach-Tarski paradoxe,
{\sl Cambridge University Press}
{\bf } (1985),


\bigskip
\end{thebibliography}
\end{document}